\documentclass[12pt]{amsart}
\usepackage[top=1in, bottom=1in, left=1in, right=1in]{geometry}

\usepackage{graphicx,color}
\usepackage{amsmath, amsthm, amsfonts, amssymb}
\usepackage{enumitem}
\usepackage[hyphens]{url}

\theoremstyle{plain}
\newtheorem{thrm}{Theorem}[section]
\newtheorem{lmm}[thrm]{Lemma}
\newtheorem{prpstn}[thrm]{Proposition}
\newtheorem{corollary}[thrm]{Corollary}

\numberwithin{sblmm}{thrm}

\theoremstyle{remark}
\newtheorem*{rmk}{Remark}

\numberwithin{equation}{section}

\newcommand{\Z}{\mathbb{Z}}

\renewcommand{\P}{\mathbb{P}}

\newcommand{\CD}{\mathcal{D}}

\newcommand{\CN}{\mathcal{N}}

\renewcommand{\phi}{\varphi}

\newcommand{\bs}\boldsymbol{}

\newcommand{\eq}[2]{ \begin{equation} \label{#1}\begin{split} #2 \end{split} \end{equation} }

\newcommand{\als}[1]{\begin{align*} #1 \end{align*} }

\renewcommand{\mod}[1]{\,({\rm mod}\,#1)}
\definecolor{brown}{rgb}{0.5,0.3,0.2}
\definecolor{orange}{rgb}{0.7,0.3,0}
\definecolor{blue}{rgb}{.2,.6,.75}
\definecolor{green}{rgb}{.4,.7,.4}

 \title[Best possible densities]{Best possible densities of Dickson $m$-tuples, as a consequence of Zhang-Maynard-Tao}
 
\author[A. Granville]{Andrew Granville}
\address{AG: D\'epartement de math\'ematiques et de statistique\\
Universit\'e de Montr\'eal\\
CP 6128 succ. Centre-Ville\\
Montr\'eal, QC H3C 3J7\\
Canada}
\email{andrew@dms.umontreal.ca}

\author[D. M. Kane]{Daniel M. Kane}
\address{DMK: Department of Mathematics \\ University of California-San Diego \\ 9500 Gilman Drive \#0404 \\ La Jolla, CA 92093 \\ USA}
\email{dakane@math.ucsd.edu}

\author[D. Koukoulopoulos]{Dimitris Koukoulopoulos}
\address{DK: D\'epartement de math\'ematiques et de statistique\\
Universit\'e de Montr\'eal\\
CP 6128 succ. Centre-Ville\\
Montr\'eal, QC H3C 3J7\\
Canada}
\email{koukoulo@dms.umontreal.ca}

\author[R. J. Lemke Oliver]{Robert J. Lemke Oliver}
\address{RJLO: Department of Mathematics \\ Stanford University \\ Building 380 \\ Stanford, CA 94305 \\ USA}
\email{rjlo@stanford.edu}

\thanks{The first and third authors are supported by Discovery Grants from the Natural Sciences and Engineering Research Council of Canada. The second and fourth authors were supported by NSF Mathematical Sciences Postdoctoral Research Fellowships.}

\begin{document}

\date{\today}

%\subjclass[2010]{ }
%\keywords{     }

\begin{abstract} We determine for what proportion of integers $h$  one now knows that there are infinitely many prime pairs $p,\ p+h$ as a consequence of the Zhang-Maynard-Tao theorem.   We consider the natural generalization of this to $k$-tuples of integers, and we determine the limit of what can be deduced assuming only the Zhang-Maynard-Tao theorem.
 \end{abstract}

\maketitle

%\setcounter{tocdepth}{1}
%\tableofcontents

%%%%%%%%%%%%%%%%%%%%%%%%%%%%%%%%%%%%%%%%%%%%%%%%%%
%%%%%%%%%%%%%%%%%%%%%%%%%%%%%%%%%%%%%%%%%%%%%%%%%%
%%%%%%%%%%%%%%%%%%%%%%%%%%%%%%%%%%%%%%%%%%%%%%%%%%

\section{Introduction and statement of results}
The twin prime conjecture states that there are infinitely many pairs of integers $(n,n+2)$ which are simultaneously primes. More generally, Hardy and Littlewood conjectured that each entry of the $k$-tuple $(n+h_1,\cdots,n+h_k)$ should be prime infinitely often, unless there is a trivial reason why this cannot happen. This ``trivial reason" is about divisibility by small primes $p$. For example, the triplet $(n+2,n+4,n+6)$ can never all be prime if $n>1$ because at least one of them must be a multiple of 3. So we call a $k$-tuple \emph{admissible} if, for each prime $p$, the reductions of the numbers $h_1,\dots,h_k$ modulo $p$ do not cover all of $\Z/p\Z$. With this definition in hand, the Hardy-Littlewood conjecture states that if $(h_1,\dots,h_k)$ is an admissible $k$-tuple, then there are infinitely many integers $n$ for which the numbers $n+h_1,\dots,n+h_k$ are all prime.

Even though we are still far from proving the full Hardy-Littlewood conjecture, there has been remarkable progress made towards it recently. Firstly, in May 2013, Yitang Zhang \cite{zhang} made headlines by proving that there are bounded gaps between primes, and specifically that
\[
\liminf_{n\to\infty} p_{n+1}-p_n < 70{,}000{,}000,
\]
where $p_n$ denotes the $n$-th prime. Then, in November 2013, James Maynard \cite{maynard} and Terence Tao independently showed, using somewhat different techniques, that $70{,}000{,}000$ can be replaced by $600$, and, as it stands right now, the best bound known is $246$, due to the Polymath project \cite{polymath8b}. Even more impressively, they proved that for any integer $m\ge1$ there is an integer $k=k_m$ such that if $(h_1,\dots,h_k)$ is an admissible $k$-tuple, then there are infinitely many integers $n$ for which at least $m$ of the numbers $n+h_1,\dots,n+h_k$ are prime. Obviously, $k_m\ge m$, and in \cite{polymath8b} it was shown that one can take $k_2=50$ and $k_m\ll e^{3.82 m}$.

We call a $k$-tuple of integers $(h_1,\dots,h_k)$ a \emph{Dickson $k$-tuple} if there are infinitely many integers $n$ for which $n+h_1,\ldots,  n+h_k$ are each prime. The Hardy-Littlewood conjecture is equivalent to the statement that ``\emph{all admissible $k$-tuples of integers are Dickson $k$-tuples}'', and the Maynard-Tao theorem implies that ``\emph{every admissible $k_m$-tuple of integers contains a Dickson $m$-tuple}''.  In particular, the Maynard-Tao theorem implies that Dickson $m$-tuples exist, yet no explicit example of a Dickson $m$-tuple is known!  Nevertheless, a simple counting argument (which also appeared in \cite{gran}) yields the following attractive consequence of the Maynard-Tao theorem.

%%%%%%%%%%%%%%%%%%%%%%%%%%%%%%%%%%%%%%%%%%%%%%%%%%
%%%%%%%%%%%%%%%%%%%%%%%%%%%%%%%%%%%%%%%%%%%%%%%%%%
%%%%%%%%%%%%%%%%%%%%%%%%%%%%%%%%%%%%%%%%%%%%%%%%%%

\begin{corollary} \label{cor1} A positive proportion of $m$-tuples of integers are Dickson $m$-tuples.
\end{corollary}

%%%%%%%%%%%%%%%%%%%%%%%%%%%%%%%%%%%%%%%%%%%%%%%%%%
%%%%%%%%%%%%%%%%%%%%%%%%%%%%%%%%%%%%%%%%%%%%%%%%%%
%%%%%%%%%%%%%%%%%%%%%%%%%%%%%%%%%%%%%%%%%%%%%%%%%%

\begin{proof}
Let $k=k_m$, so that $m\leq k$, and define $R=\prod_{p\leq k} p$ and $x=NR$ for some (very large) integer $N$. We let
\[
\CN=\{ n\in (-x,x]:\ (n,R)=1\} ,
\]
so that $|\CN| =2x\phi(R)/R$. Any subset of $k$ elements of $\CN$ is admissible, since it does not contain any integer $\equiv 0 \pmod p$ for any prime $p\leq k$.  There are $\binom{ |\CN|}{k}$ such $k$-tuples. Each contains a Dickson $m$-tuple by the Maynard-Tao theorem.

Now suppose that there are $T(x)$ Dickson $m$-tuples within $\CN$. Any such $m$-tuple is a subset of exactly $\binom{ |\CN|-m}{k-m}$ of the $k$-subsets of $\CN$, and hence
\[
T(x) \cdot \binom{ |\CN|-m}{k-m} \geq \binom{ |\CN| }{k},
\]
and therefore
\[
T(x)\geq  \binom{|\CN|}k \left/ \binom {|\CN|-m}{k-m} \right.
	 \geq (|\mathcal N|/k)^m=\left(\frac{\phi(R)}{kR}\right)^m \cdot (2x)^m ,
\]
as desired. 
\end{proof}

%%%%%%%%%%%%%%%%%%%%%%%%%%%%%%%%%%%%%%%%%%%%%%%%%%
%%%%%%%%%%%%%%%%%%%%%%%%%%%%%%%%%%%%%%%%%%%%%%%%%%
%%%%%%%%%%%%%%%%%%%%%%%%%%%%%%%%%%%%%%%%%%%%%%%%%%

The main goal of this paper is to provide better lower bounds on the proportion of $m$-tuples that are Dickson $m$-tuples. If $\Delta(m)$ is the proportion of such $m$-tuples, then Corollary \ref{cor1} implies that $\Delta(m)>0$. We are interested in determining the best possible lower bound on $\Delta(m)$ assuming only the results of Zhang, Maynard, and Tao and treating them as ``black boxes'':\ If $\CD_m$ is the set of Dickson $m$-tuples, then  there is an integer $k=k_m$ for which:
\begin{itemize}
\item $\CD_m$ is translation-invariant, that is to say, if $(h_1,\dots,h_m)\in\CD_m$ and $t\in\Z$, then $(h_1+t,\dots,h_m+t)\in\CD_m$;
\item $\CD_m$ is permutation-invariant, that is to say, if $(h_1,\dots,h_m)\in\CD_m$ and $\sigma\in S_m$, then $(h_{\sigma(1)},\dots,h_{\sigma(m)}) \in\CD_m$;
\item for any admissible $k$-tuple, $(x_1,\dots,x_{k})$, there exist distinct $h_1,\ldots,h_m \in \{x_1,\ldots,x_{k}\}$ such that $(h_1,\ldots,h_m)\in \CD_m$.
\end{itemize}
We call any set $A\subset\Z^m$ that has the above properties \emph{$(m,k)$-plausible}. Then we define
\eq{eqn:delta_def}{
\delta(m,k) = \min\left\{  \liminf_{N\rightarrow\infty} \frac{|A\cap [-N,N]^m|}{(2N)^m}   : A\subset \Z^m,\ A\ \text{is} \ (m,k)-\text{plausible} \right\} ;
}
and, for any $m\ge 1$, we have that
\[
\Delta(m)\ge \delta(m,k_m)
\]
by the Zhang-Maynard-Tao theorem. Our main result is the following.

%%%%%%%%%%%%%%%%%%%%%%%%%%%%%%%%%%%%%%%%%%%%%%%%%%
%%%%%%%%%%%%%%%%%%%%%%%%%%%%%%%%%%%%%%%%%%%%%%%%%%
%%%%%%%%%%%%%%%%%%%%%%%%%%%%%%%%%%%%%%%%%%%%%%%%%%

\begin{thrm}\label{DensityTheorem}
For $k\geq m\geq 1$, we have, uniformly
$$
\delta(m,k) =   \frac{(\log 2m)^{O(m)} }{    ( \frac{k}{m} \log\log \frac{3k}{m} )^{m-1}}.
$$
\end{thrm}

%%%%%%%%%%%%%%%%%%%%%%%%%%%%%%%%%%%%%%%%%%%%%%%%%%
%%%%%%%%%%%%%%%%%%%%%%%%%%%%%%%%%%%%%%%%%%%%%%%%%%
%%%%%%%%%%%%%%%%%%%%%%%%%%%%%%%%%%%%%%%%%%%%%%%%%%

In the Maynard-Tao Theorem we know that one can obtain $k_m\leq e^{cm}$ for some constant $c>0$.
The best value known for $c$ is a little smaller than $3.82$, and Tao \cite{tao, gran} showed
that the Maynard-Tao technique cannot be (directly) used to obtain a constant smaller than $2$.  So we deduce that
\[
\Delta(m)  \geq \delta(m,k_m) \geq    e^{-(c+o(1))m^2}\gg  e^{-4m^2},
\]
and that this can, at most, be improved to $\gg e^{-2m^2}$ using the currently available methods.

A \textsl{de Polignac  number} is an integer $h$ for which there are infinitely many pairs of primes $p$ and $p+h$ (the twin prime conjecture is the special case $h=2$).  A positive proportion of integers are de Polignac numbers, as follows from both results above, but we wish to determine the best explicit lower bounds possible.  In Section \ref{sec:depolignac}, using quite elementary arguments, we will show that
\begin{equation} \label{eqn:dP_lowerbound}
\Delta(2)
	\geq \delta(2,k_2) \geq \frac{1}{49} \ \prod_{p\leq 50} \left( 1 -\frac 1p\right) 
	\approx 0.002830695767 \ldots 
	> \frac 1{354},
\end{equation}
and that this cannot be improved dramatically without improving the value of $k_2=50$.  We remark that this bound is surprisingly good in light of the fact that we can only deduce that there is some $h\leq 246$ for which there are infinitely many prime pairs $p,p+h$.

\medskip

Finally, we note that it would perhaps be more natural to consider the proportion $\Delta^\mathrm{ad}(m)$ of \emph{admissible} $m$-tuples that are Dickson $m$-tuples, rather than the proportion of all $m$-tuples.  To this end, in Section \ref{sec:adm_prop}, we determine the proportion $\rho_{\mathrm{ad}}(m)$ of all $m$-tuples that are admissible.  Surprisingly, this question has seemingly not been addressed in the literature, and we prove that
\begin{equation} \label{eqn:adm_prop}
\rho_{\mathrm{ad}}(m) = \frac{e^{o(m)}}{(e^\gamma \log m)^m}
\end{equation}
as $m\to\infty$. 

%To do this we introduce a lower bound on
%$\delta(2,k)$:\

%Let $B$ be a maximal admissible set that contains no Dickson pair (so that $1\leq |B|\leq k_2-1$), such that if $B\cup \{ t\}$ is admissible (where $t\not\in B$), then $B\cup \{ t\}$ contains a Dickson pair. Then $t-B=\{ t-b:\ b\in B\}$ contains a  de Polignac number.  So define $\eta(2,\ell)$ as follows: For any given admissible set with $\ell$ elements let $A$ be the subset of the integers with smallest lower density, $\eta(\ell)$, such that  $t-B$ contains an element of $A$ whenever $B\cup \{ t\}$ is admissible. Then
%\[
%\Delta(2)  \geq \delta(2,k_2) \geq \min_{1\leq \ell\leq k_2-1} \eta(\ell) .
%\]
%We shall prove that
%\[
%\eta(\ell)  \sim
%\frac{e^{-\gamma}}{\ell \log \ell}  ;
%\]
%and that, explicitly,
%\[
%\Delta(2) \geq \frac 1{49} \ \prod_{p\leq 49} \left( 1 -\frac 1p\right) \approx 0.002830695767 \ldots > \frac 1{354},
%\]

\section{The density of de Polignac numbers}
\label{sec:depolignac}

To prove a lower bound on the density of de Polignac numbers, we consider admissible sets $B$ of integers that do not contain a Dickson pair.  If the prime $k$-tuplets conjecture is true, then necessarily $|B|=1$, while Zhang's theorem implies that $|B|\leq k_2-1$.  Because of this upper bound, there must be maximal such sets, in that, for any $t\not\in B$ such that $B\cup \{ t\}$ is admissible, $B\cup \{ t\}$ contains a Dickson pair. This condition implies that
$t-B$ contains a de Polignac number, and we will obtain a lower bound on $\delta(2,k_2)$ by varying $t$.  To this end, for an admissible set $B$, define $\eta(B)$ to be the minimal lower density of sets $A$ with the property that $t-B$ contains an element of $A$ whenever $B\cup\{t\}$ is admissible.  Moreover, for any integer $\ell$, set $\eta(\ell)$ to be the infimum of $\eta(B)$ as $B$ runs over admissible sets of size $\ell$.  We thus have that
\[
\Delta(2) \geq \delta(2,k_2) \geq \min_{1\leq \ell \leq k_2-1} \eta(\ell),
\]
and we will prove the following.

\begin{prpstn} \label{prop:eta}
For any integer $\ell$, we have that $\eta(\ell) \sim e^{-\gamma}/\ell\log\ell$, and explicitly that
\[
\frac{1}{\ell} \prod_{p\leq \ell+1} \left(1-\frac{1}{p}\right) \leq \eta(\ell) \leq \frac{1}{\ell-y} \prod_{p\leq y} \left(1-\frac{1}{p}\right)
\]
for any positive integer $y\leq \ell-1$.
\end{prpstn}

For the particular application to $\Delta(2)$, we take $\ell\in\{1,\dots,49\}$ in Proposition \ref{prop:eta} to find that
\[
\delta(2,k_2) \geq \min_{1\le\ell\le 49} \eta(\ell) > 0.002830695767 > \frac{1}{354},
\]
while taking $y=13$ yields that $\eta(49) < 0.005328005328 < \frac{1}{187}$, so that, using our techniques, we can hope for an improvement in our lower bound for $\delta(2,k_2)$ by a factor of at most about two.  
However, we can do better: the explicit upper bound given in Lemma \ref{lem:explicit_upper} below shows that $\delta(2,50) \leq \frac{1}{210} < 0.0048$.

\begin{proof}[Proof of Proposition \ref{prop:eta}]
We begin with the lower bound for $\eta(\ell)$.  Suppose that $B$ is admissible of size $\ell$, so that for each prime $p\leq \ell+1$ there exists a residue class $n_p\pmod p$ such that $p\nmid n_p+b$ for each $b\in B$. If $t\not \equiv -n_p \pmod p$ for all $p\leq \ell+1$ then
$B\cup \{ t\}$ is admissible and so contains a Dickson pair, and, as remarked above, $t$ must be one of that pair, else $B$ would have contained a Dickson pair. If $x$ is large then the number of such integers $t\leq x$ is
\[ \sim \prod_{p\leq \ell+1} \left( 1 -\frac 1p\right) x \]
and we know that, for some $b\in B$, $t-b$ is a de Polignac number.  Given $t\in\Z$, an integer can be written in at most $\ell$ ways as $t-b$ for $b\in B$, since $|B|=\ell$. Hence,
\[
\eta(B) \geq \frac 1\ell \  \prod_{p\leq \ell+1} \left( 1 -\frac 1p\right) \sim
\frac{e^{-\gamma}}{\ell \log \ell} 
\]
as $\ell\to\infty$. 
%More importantly as $\ell\leq 49$ we deduce that
%\[
%\Delta(2) \geq \delta(2,50) \geq \min_{1\leq \ell\leq 49} \eta(\ell) \geq \frac 1{49} \ \prod_{p\leq 49} \left( 1 -\frac 1p\right) \approx 0.002830695767 \ldots > \frac 1{354}.
%\]
%Since we can only have even numbers, this tells us that for at least $\frac 12 \% $ of even integers $h$, there are infinitely many pairs of primes $p,p+h$. This bound is surprisingly good given that we can only deduce that there is some $h\leq 246$ for which there are infinitely many pairs of primes $p,p+h$.

%The values of $t$ that occur in the proof are periodic, of period $r:=\prod_{p\leq k_2} p$, so that there must be a de Polignac number $t-b$ in every interval of length $C:=r+\max_{b\in B} b- \min_{b\in B} b$. This is a different proof of a result of Pintz \cite{pintz-polignac}.

%\section{Upper bound on $\eta(\ell)$}
We now turn to the upper bound.  We will construct sets $A$ and $B$ with the properties that $B$ is admissible and that if $t$ is such that $B\cup \{ t\}$ is admissible, then there exists $b=b_t\in B$ such that $t-b_t\in A$.

Let $y\in\{1,\dots,\ell-1\}$ and set $v=\ell-y$. Moreover, define $r=\prod_{p\leq y}p$ and $m=\prod_{p\leq \ell} p$, and select $h$ large so that $q:=hv+1$ is a prime $>\ell$.  We define $A$ to be the set of all integers $a$ for which $(a+1,r)=1$ and $a\equiv 0,1,2,\ldots, $ or $h-1 \pmod q$.  The density $\delta(A)$ of $A$ is
\[
\frac h q \prod_{p\leq y} \left( 1 -\frac 1p\right)
\le \frac{1}{\ell-y} \prod_{p\leq y} \left( 1 -\frac 1p\right)  .
\]
Our set $B$ will be constructed as the union of two sets, $B_1$ and $B_2$.  For $B_1$, we take $B_1:=\{ b_0,b_1,\ldots ,b_v\}$  where each $b_i$ is chosen to satisfy $b_i\equiv ih \pmod q$ and $b_i\equiv 1 \pmod m$, and we choose $B_2$ to be a set of $y-1$ integers covering the residue classes $2,\dots,p-1 \pmod{p}$ for each prime $p\leq y$. We note that $B:=B_1\cup B_2$ has $\ell$ elements and is admissible: it occupies the congruence classes $1, \dots, p-1 \pmod p$ for all $p\leq y$ and covers at most $y$ classes modulo $p$ for $p>y$.

If $t$ is such that $B\cup \{ t\}$ is admissible, then $t\not\equiv 0 \pmod p$ for all primes $p\leq y$. We can write
$t\equiv ih+j \pmod q$ for some $i$ and $j$ satisfying $0\leq i\leq v$ and $0\leq j\leq h-1$. Consider
$t-b_i$. This is $\equiv j \pmod q$ and $\not\equiv -1 \pmod p$ for all primes $p\leq y$. Hence
$t-b_i\in A$, so that $\eta(\ell) \geq \delta(A)$, and the result follows.
\end{proof}

\begin{rmk}
Pintz \cite{pintz-polignac} showed that there is an ineffective constant $C$ such that every interval of length $C$ contains a de Polignac number.  The proof of the lower bound above furnishes a different proof of this result: the values of $t$ that appear are periodic modulo $r:=\prod_{p\leq k_2}p$, so, given a maximal set $B$, we obtain Pintz's result with $C=r+\max_{b\in B} b- \min_{b\in B} b$.
\end{rmk}

% The density of $A$ is
%\[
%\frac h q \prod_{p\leq y} \left( 1 -\frac 1p\right)
%\sim
%\frac{e^{-\gamma}}{\ell \log \ell}
%\]
%as $\ell\to\infty$. If $\ell=49$ we minimize the upper bound taking $y=13$ to get the upper bound $\leq 0.005289801208>1/189$, which is bigger by a factor of about two
%\end{proof}

\section{Proof of Theorem \ref{DensityTheorem}}

%In order to  prove Theorem \ref{DensityTheorem}, we must first give a precise definition of $\delta(m,k)$ for $m\leq k$.  Given a subset $A\subset \Z^m$ and $h\in\mathbb{R}^m$, we define
%\[
%A+h = \{a+h:a\in A\},
%\]
%and, given a permutation $\sigma\in S_m$, we define
%\[
%A^\sigma = \{(a_{\sigma(1)},\dots,a_{\sigma(m)}): (a_1,\dots,a_m)\in A\} .
%\]
%We say that $A$ is \textsl{closed under translations} if $A+h=A$ for all $h\in\mathbb{Z}^m$ and that $A$ is \textsl{closed under permutations} if $A^\sigma=A$ for all $\sigma\in S_m$. A subset $A\subset \Z^m$ is \textsl{$(m,k)$-plausible} if it is closed under translations and permutations, and if for every $k$-admissible set $\{x_1,\ldots,x_k\}$, there exist distinct $y_1,\ldots,y_m \in \{x_1,\ldots,x_k\}$ for which $(y_1,\ldots,y_m)\in A$.  We now define $\delta(m,k)$ as
%\begin{equation}\label{eqn:delta_def}
%\delta(m,k):= \ \min_{\substack{A\subset \Z^m \\ A\ \text{is} \ (m,k)-\text{plausible}}}
% \liminf_{N\rightarrow\infty} \frac{|A\cap [-N,N]^m|}{(2N)^m}
%\end{equation}

We begin with a strong version of the upper bound of Theorem \ref{DensityTheorem}.
\begin{lmm} \label{lem:explicit_upper} If $k\ge m\ge2$ with $k/m\to\infty$, then 
\[
\delta(m,k) \leq   \left(\frac{e^{-\gamma}+o(1)}{\frac{k}{m-1}\log\log\frac{k}{m-1} } \right)^{m-1}  .
\]
%\dmcom{Made this a bit more precise.}
\end{lmm}

\begin{proof}
Pick $q$ as large as possible so that $(m-1)\phi(q) < k$. We claim that
$$
A = \{(x_1,\ldots,x_m):x_1\equiv x_2 \equiv \ldots \equiv x_m \mod{q}\}
$$
is $(m,k)$-plausible. This is because any admissible set with $k$ elements may take at most $\phi(q)$ distinct values modulo $q$, so by the pigeonhole principle, at least $m$ of these values must be congruent modulo $q$, giving an $m$-tuple contained in $A$.  On the other hand, $q$ is of size $\geq (e^\gamma+o(1))\frac{k}{m-1}\log\log\frac{k}{m-1}$ as $k/m\to\infty$, and $A$ has density $1/q^{m-1}$. This completes the proof.
\end{proof}

The proof of the lower bound is somewhat more involved. 
% Explicitly, our goal is to prove the following.
%
%\begin{prpstn} \label{mainresult} If $A$ is an $(m,k)$-plausible set with $1<m<k$ then
%\[
%\liminf_{x\to\infty} \frac{|A\cap [-2x,2x]^m|}{(2x)^m}  \ge   \frac{  (\log 3m)^{O(m)} }{   ( \frac km\log\log \frac km)^{m-1}}  .
%\]
%\end{prpstn}
%
A key ingredient is the Lov\'asz Local Lemma:
\medskip

\noindent \textbf{Lov\'asz Local Lemma.}\ \textit{ Suppose that $E_1,\ldots, E_n$ are events,  each of which occurs with probability $\leq p$ and depends on no more than $d$ of the others. If $d\leq 1/(4p)$ then the probability that no $E_j$ occurs is at least $e^{-2pn}$.}
\medskip

The input for the Lov\'asz Local Lemma will come from the following technical result.

\begin{lmm}\label{TechLemm} 
Let $k\ge m\ge1$. Consider a translation and permutation invariant set $A\subset ((-2x,2x]\cap \Z)^m$, and set $\CN=\{n\in(-x,x]\cap \Z: (n,\prod_{p\le k}p)=1 \}$. If
\[
\frac{|A\cap \CN^m|}{|\CN^m|} > \frac{1}{8 m \binom{k-1}{m-1}},
\]
and $x$ is large enough in terms of $m$ and $k$, then 
\[
|A| \ge  \frac{ (\log 3m)^{O(m)} }{  ( \frac km\log\log \frac km)^{m-1}} \cdot x^m .
\]
\end{lmm}

Before we prove Lemma \ref{TechLemm}, we use it to deduce the lower bound in Theorem \ref{DensityTheorem}.

\begin{proof}[Proof of the lower bound in Theorem \ref{DensityTheorem}] Fix a large positive $x$ and suppose that $A\subset [-2x,2x]^m$ is an $(m,k)$-plausible set, yet $|A| \le (\log 2m)^{-cm}  x^{m}  \big/  ( \frac km\log\log \frac km)^{m-1} $ for some $c>0$. If $c$ is large enough, then Lemma \ref{TechLemm} implies that
\eq{A-hyp}{
\frac{|A\cap \CN^m|}{|\CN^m|}   
   	\leq  \frac{1}{8 m \binom{k-1}{m-1}},
}
where $\CN:=\{n\in(-x,x]\cap \Z: (n,\prod_{p\le k}p)=1 \}$. Now select integers $n_1,\ldots,n_k$ uniformly at random from $\CN$. We claim that there is a positive probability that the $n_i$ are distinct and that there is no $m$-element subset of $\{n_1,\dots,n_k\}$ in $A$, which implies that $A$ is not $(m,k)$-plausible.  We use the Lov\'asz Local Lemma, in which we wish to avoid the events $n_i=n_j$ for $1\leq i<j\leq k$ and $(x_1,\ldots,x_m)\in A$ for any choice of $m$ elements $x_1,\dots,x_m$ among the integers $n_1,\ldots,n_k$. There are $\binom{k}{2}+\binom{k}{m}$ such events. Each event depends on no more than $m$ of the $n_i$, and so depends on no more than $m\binom{k-1}{1}+m\binom{k-1}{m-1}\leq 2m\binom{k-1}{m-1}$ other events. If $x_1,\dots,x_m$ are $m$ out of the $k$ random variables, then the probability that $(x_1,\ldots,x_m)\in A$ is $\le1/(8m\binom{k-1}{m-1})$ by \eqref{A-hyp}. We finally note that the probability that $n_i=n_j$ is $1/|\CN| \sim R/(2x\phi(R))$ as $x\to\infty$, which is certainly $\leq 1/(8 m \binom{k-1}{m-1})$ for $x$ sufficiently large.

The Lov\'asz Local Lemma  now implies that there exists a subset $n_1,\ldots,n_k$ of distinct elements of $\mathcal N$ for which $(x_1,\ldots,x_m)\not\in A$ for all subsets $\{x_1,\ldots,x_m\}$ of $\{n_1,\ldots,n_k\}$. (In fact, the Lov\'asz Local Lemma implies that this is true for a proportion $\geq e^{-k/2m^2}$ of the $k$-subsets of $\mathcal N$.)  Hence, $A$ is not an $(m,k)$-plausible set, a contradiction.
\end{proof}

To prove Lemma \ref{TechLemm}, we first need another result. Given an $m$-tuple $\bs h=(h_1,\dots,h_m)$, we denote by $n_p(\mathbf{h})$ the number of congruence classes mod $p$ covered by $h_1,\dots,h_m$. Note that $1\leq n_p(\bs h)\leq \min\{ p,m\}$ and that both upper and lower bounds are easily obtained for some $m$-tuple $\bs h$.

\begin{lmm}\label{TechLemm2} Let $k\ge m\ge1$. There is an absolute constant $c>0$ such that if $\bs h$ is a randomly selected $m$-tuple from $\CN=\{n\in(-x,x]\cap \Z: (n,\prod_{p\le k}p)=1 \}$ and $x$ is large enough in terms of $m$ and $k$, then the probability that
\[
\prod_{p\leq k} \left(1-\frac{n_p(\bs h)}{p}\right) \left(1-\frac{1}{p}\right)^{-m}
> (\log 3m)^{cm} (\log\log k)^{m-1}
\]
is
\[
 \leq \frac{1}{16 m \binom{k-1}{m-1}} .
\]
\end{lmm}

\begin{proof} The result is trivial if $k\le m^2$, so we will assume that $k> m^2$. Similarly, we may assume that $k$ is large enough. Moreover, note that there are $\ll_k |\CN|^{m-1}=o_{x\to\infty} (|\CN|^m)$ $m$-tuples $(h_1,\dots,h_m)$ with non-distinct elements. So it suffices to show that if we randomly select a subset $B=\{ b_1,\ldots,b_m\}$ of distinct elements of $\CN$, then the probability that
\[
\prod_{p\leq k} \left(1-\frac{n_p(B)}{p}\right) \left(1-\frac{1}{p}\right)^{-m}
> (\log 3m)^{cm} (\log\log k)^{m-1}
\]
is
\[
 \leq \frac{1}{17 m \binom{k-1}{m-1}} .
\]

By Mertens' Theorem, the contribution of the small primes is
\[
\prod_{p\leq m^2} \left(1-\frac{n_p(B)}{p}\right) \left(1-\frac{1}{p}\right)^{-m} \leq
\prod_{p\leq m^2}  \left(1-\frac{1}{p}\right)^{-m} \leq ( \log 3m)^{c_1m}
\]
for some absolute constant $c_1>0$. If $c\ge2c_1$ and we set
\[
f(B)=\prod_{m^2<p\leq k} \left(1-\frac{n_p(B)}{p}\right) \left(1-\frac{1}{p}\right)^{-m},
\]
then we are left to show that
\[
\P\Big(f(B) > (\log 3m)^{cm/2} (\log\log k)^{m-1} \Big) \le \frac{1}{17 m \binom{k-1}{m-1}} .
\]
Taking logarithms, we have 
\begin{align*}
\log f(B) &\le \sum_{m^2<p\leq k} \left( \frac{m-n_p}p + O\left( \frac{m^2}{p^2} \right) \right)
\leq  \sum_{m^2<p\leq k  }  \frac{m-n_p}p + O\left( \frac{1}{\log m} \right)  \\
&\leq  (m-1) \sum_{m^2<p\leq k}  \frac{X_p(B)}p + O\left( \frac{1}{\log m} \right) ,
\end{align*}
where $X_p(B)=1$ if $n_p(B)<m$, and $X_p(B)=0$ if $n_p(B)=m$. Therefore, if we set $X(B) = \sum_{m^2<p\le k} X_p(B)/p$ and we assume that $c$ is large enough, then it suffices to show that
\[
\P\left( X(B) > \frac{c}{3}\log\log(3m) +\log\log\log k \right) \le \frac{1}{17 m \binom{k-1}{m-1}} .
\]
In turn, the above inequality is a consequence of the weaker estimate
\eq{techlemma-e1}{
\P\left( X(B) > \log\log(\max\{m^2,m\log k\}) + c' \right) \le \frac{1}{17 m \binom{k-1}{m-1}} ,
}
where $c'$ is a sufficiently large constant. 

The $X_p$ can be viewed as independent random variables as we run over all possible sets $B$. As in the birthday paradox, the probability that $X_p=0$ is
\[
\left(1-\frac{1}{p-1}\right)\left(1-\frac{2}{p-1}\right)\ldots \left(1-\frac{m-1}{p-1}\right) =
\exp \left( -\frac{m^2+O(m)}{2p} \right)= 1+O\left(\frac{m^2}{p} \right)    .
\]
For any $r$, if $p\leq r$, then we trivially have that $\mathbb{E}[e^{rX_p/p}] \leq e^{r/p}$, and otherwise
\als{
\mathbb{E}[e^{rX_p/p}] = \mathbb{P}(X_p=0) +  \mathbb{P}(X_p=1) e^{r/p}
&= 1 +  \mathbb{P}(X_p=1) (e^{r/p}-1) \\
&\leq \exp\left(  O\left( \frac{m^2r}{p^2}     \right) \right).
} 
Therefore, for any values of $s$ and $r$, we have that
\begin{align*}
e^{rs}\mathbb{P}\left(X \geq s\right) & \leq \mathbb{E}[e^{rX}] = \prod_{p\leq k} \mathbb{E}[e^{rX_p/p}] \leq \prod_{p\leq r} e^{r/p} \prod_{p\geq r} e^{O(  {m^2r}/{p^2} ) }\\
& \leq \exp\left( r\log\log r + O(r) + O(m^2/\log r)               \right).
\end{align*}
Thus, if $r\ge m^2$, then setting $s=\log\log r+c'$ for $c'$ sufficiently large, we find that
$$
\mathbb{P}\left(X \geq \log\log r+c' \right) \leq e^{-r} .
$$
Substituting $r=\max\{m^2,m\log k\}$ establishes \eqref{techlemma-e1} for $k$ large enough, thus completing the proof of the lemma.
\end{proof}

\begin{proof} [Proof of Lemma \ref{TechLemm}] We partition the elements of $[-2x,2x]^m$ into \emph{translation classes}, putting two elements in the same class if and only if they differ by $(\ell,\ell,\ldots,\ell)$ for some integer $\ell$. Each translation class $T$ intersecting $[-x,x]^m$ contains at least $2x$ elements of $[-2x,2x]^m$ (and at most $6x$). The main idea of the proof is that if we can find at least $U$ elements of $A\cap\CN^m$ whose translation class has at most $M$ elements inside $\CN^m$, then $|A|\ge 2x\cdot U/M$.

Note that $n_p(\bs h)$ is fixed over all $\bs h\in T$, so we denote it by $n_p(T)$. The number of integers $\ell$ for which $\mathbf{h}+(\ell,\ell,\ldots,\ell) \in \mathcal N^m$ is $\leq 3x \prod_{p\leq k}\left(1-\frac{n_p(T)}{p}\right)$ for $x$ large enough. If we set $R=\prod_{p\le k}p$, then Lemma \ref{TechLemm2} implies that the proportion of elements of $\CN^m$ for which 
\eq{translation class bound}{
\prod_{p\leq k} \left(1-\frac{n_p(B)}{p}\right)
\leq \left(\frac{\phi(R)}{R} \right)^m (\log 3m)^{cm} (\log\log k)^{m-1}
}
is $ \le  {1}\big/ 16 m \binom{k-1}{m-1}$, provided that $x$ is large enough, with $c$ being an absolute constant. Since $A$ contains at least $|\CN^m| \big / 8m\binom{k-1}{m-1}$ elements in $\CN^m$ by assumption, we find that $A$ contains at least $|\CN^m| \big / 16m\binom{k-1}{m-1}$ for which \eqref{translation class bound} holds. Therefore, we conclude that
\[
|A| \ge \frac{|\CN^m| }{ 24 m \binom{k-1}{m-1}} 
	\left/ \left\{ \left(\frac{\phi(R)}{R} \right)^m (\log 3m)^{cm} (\log\log k)^{m-1} \right\} \right.
> \frac{x^m    (\log 3m)^{O(m)}}{    ( \frac km\log\log \frac km)^{m-1}}  
\]
for $x$ large enough, and the result follows.
\end{proof}

\section{The number of admissible $k$-tuples}
\label{sec:adm_prop}

%\acom{I think this section is interesting but I am not sure where to put it. Also it is not exactly a proof but maybe someone can see how to make it into one?}

Our goal in this section is to show relation \eqref{eqn:adm_prop}. Given a prime $p$, we say that an $m$-tuple is \emph{admissible mod $p$} if its elements do not occupy all of the residue classes mod $p$, so an $m$-tuple is admissible if and only if it is admissible mod $p$ for every prime $p$. By the pigeonhole principle, any set of $m$ integers is admissible mod $p$ if $p>m$, so to test for admissibility we need only work with the primes $p\leq m$. This implies, using the Chinese Remainder Theorem, that 
\[
\rho_{\text{ad}}(m) = \prod_{p\le m} \rho_{\text{ad}}(m,p),
\]
where $\rho_{\text{ad}}(m,p)$ denotes the proportion of $m$-tuples that are admissible mod $p$. 

If $m/\log m<p\le m$, then we note the trivial bounds $(1-1/p)^m\le \rho_{\text{ad}}(m,p) \le 1$, with the lower bound coming from counting $m$-tuples whose elements are not $0\mod p$. Therefore
\[
1\ge \prod_{m/\log m<p\le m} \rho_{\text{ad}}(m,p) \ge \prod_{m/\log m<p\le m}\left(1-\frac{1}{p}\right)^{m}= e^{O(\frac{m\log\log m}{\log m})} .
\]
It remains to compute the contribution of primes $p\le m/\log m$. It is not difficult to determine an exact expression for $\rho_{\text{ad}}(m,p)$ using an inclusion-exclusion argument: the probability that the elements of an $m$-tuple $\bs h$ belong to a given subset of $p-1$ residue classes is $(1-\frac 1p)^m$. There are $\binom p1$ choices of the $p-1$ residue classes. If the elements of $\bs h$ belong to exactly $p-2$ residue classes mod $p$ then $\bs h$ was just counted twice and so we need to subtract the probability of this happening. That probability is $(1-\frac 2p)^k$, and there are $\binom p2$ choices of the $p-2$ residue classes. Continuing in the way, we find that
\[
\rho_{\text{ad}}(m,p) = \sum_{j=1}^{m-1}  \binom{p}{j} (-1)^{j-1} \left( 1 -\frac jp\right)^m \ = \ p\left( 1 -\frac 1p\right)^m-\binom{p}{2}   \left( 1 -\frac 2p\right)^m+\ldots
\]
We note that ratio of two consecutive summands in absolute value is
\als{
\frac{\binom{p}{j} \left( 1 -\frac jp\right)^m}{ \binom{p}{j+1} \left( 1 -\frac{j+1}p\right)^m}
	= \left(1+\frac{1}{p-j-1}\right)^m \frac{j+1}{p-j} 
	&\ge \left(1+\frac{1}{p-2}\right)^m \frac{1}{p-1} \\
	&\ge 2\exp\left\{\frac{m}{p-1}-\log(p-1)\right\} \ge 2\log m 
}
for all $p\le m/\log m$. Therefore, we deduce that
\[
\rho_{\text{ad}}(m,p) = p\left( 1 -\frac 1p\right)^m\left(1+O\left(\frac{1}{\log m}\right)\right)  ,
\]
which implies that
\[
\rho_{\text{ad}}(m)= e^{O(\frac{m\log\log m}{\log m})}  
	\prod_{p\le m/\log m} p\left( 1 -\frac 1p\right)^m
	=  \frac{e^{O(\frac{m\log\log m}{\log m})}  }{(e^\gamma\log m)^m},
\]
which proves (a quantitative version of) relation \eqref{eqn:adm_prop}.

%
%
%If $p<k/\log k$ then the above is $p(1-\frac 1p)^k(1+O(1/\log k))$, and so the product over these primes is
%$e^{o(k)}/(e^\gamma\log k)^k$. It is difficult to work with our explicit formula for larger $p$.
%
%For larger $p$, we look at individual residue classes. The  probability distribution function for the number of integers of $A$ that are in a given residue class is Poisson with parameter $k/p$, so the probability that it contains at least one element of $A$ is $(1-e^{-k/p})$. Hence the probability that every residue class mod $p$ contains an element of $A$  is roughly $(1-e^{-k/p})^p$; that is the probability that $A$ is admissible mod $p$ is roughly $1-(1-e^{-k/p})^p$. This is $1+O(1/k)$ provided $k>p>k/(\log k-2\log\log k)$.  For the remaining interval the probability is $\geq 1/\log k$ for each $p$, and so all of these larger primes together contribute $e^{o(k)}$ to the total
%
%Hence the probability that a random set of $k$ integers is admissible is $e^{o(k)}/(e^\gamma\log k)^k$.

\section{A better density in the continuous case}
Analogous to the discrete question considered here, one can also ask about the continuous version, i.e. the set of limit points $\mathcal L$ of the set of values of $(p_{n+1}-p_n)/\log p_n$, where $p_n$ is the $n$th prime.
One can deduce from a uniform version of Zhang's Theorem that for any $0\leq \beta_1\leq \beta_2\leq \ldots \leq \beta_k$ with $k=k_2$ there exists $1\leq i<j\leq k$ such that
$\beta_j-\beta_i\in \mathcal L$; in fact, this was done by Banks, Freiberg, and Maynard \cite{banks}. By a small modification of the argument used to prove
Corollary \ref{cor1}, one can then show that $\mathcal L\cap[0,T]$ has Lebesgue measure $\gtrsim T/(k-1)$.  Somewhat remarkably, Banks, Freiberg, and Maynard were able to go beyond this by showing that, given a sufficiently large $k$-tuple partitioned into $9$ equal parts, at least two of these parts must simultaneously represent a prime.  From this, they deduce that $\mathcal L\cap[0,T]$ has Lebesgue measure $\gtrsim T/8$.

%It is not difficult to see that this is more-or-less best possible, given the result used, since the set $\mathcal L_0:=\{ t:\ \| t\| \leq 1/(k-1)\}$ contains a $\beta_j-\beta_i$ for all possible choices of the $\beta$'s.

  \end{document}